\documentclass[1=pt]{article}
\usepackage{latexsym,amssymb,amsthm,enumerate,float,geometry,cite}
\newtheorem{theorem}{Theorem}

\newtheorem{corollary}[theorem]{Corollary}
\newtheorem{lemma}[theorem]{Lemma}

\usepackage{lineno}
\usepackage{setspace}

\begin{document}
\onehalfspace

\title{Maximally distance-unbalanced trees}
\author{Marie Kramer \and Dieter Rautenbach}
\date{}

\maketitle

\begin{center}
{\small 
Institute of Optimization and Operations Research, Ulm University,\\ 
Ulm, Germany, \texttt{$\{$marie.kramer,dieter.rautenbach$\}$@uni-ulm.de}}
\end{center}

\begin{abstract}
For a graph $G$, and two distinct vertices $u$ and $v$ of $G$,
let $n_G(u,v)$ be the number of vertices of $G$ that are closer in $G$ to $u$ than to $v$.
Miklavi\v{c} and \v{S}parl (arXiv:2011.01635v1)
define the distance-unbalancedness ${\rm uB}(G)$ of $G$
as the sum of $|n_G(u,v)-n_G(v,u)|$
over all unordered pairs of distinct vertices $u$ and $v$ of $G$.
For positive integers $n$ up to $15$, 
they determine the trees $T$ of fixed order $n$
with the smallest and the largest values of ${\rm uB}(T)$, respectively.
While the smallest value is achieved by the star $K_{1,n-1}$ for these $n$,
which we then proved for general $n$ 
(Minimum distance-unbalancedness of trees,
Journal of Mathematical Chemistry, 
DOI 10.1007/s10910-021-01228-4),
the structure of the trees maximizing the distance-unbalancedness remained unclear.
For $n$ up to $15$ at least,
all these trees were subdivided stars.
Contributing to problems posed by Miklavi\v{c} and \v{S}parl,
we show 
$$\max\Big\{{\rm uB}(T):T\mbox{ is a tree of order }n\Big\}
=\frac{n^3}{2}+o(n^3)$$
and 
$$\max\Big\{{\rm uB}(S(n_1,\ldots,n_k)):1+n_1+\cdots+n_k=n\Big\}
=\left(\frac{1}{2}-\frac{5}{6k}+\frac{1}{3k^2}\right)n^3+O(kn^2),$$
where $S(n_1,\ldots,n_k)$ is the subdivided star 
such that removing its center vertex leaves paths of orders $n_1,\ldots,n_k$.\\[3mm]
{\bf Keywords:} Distance-unbalancedness; distance-balanced graph; Mostar index
\end{abstract}

\section{Introduction}
Inspired by graph theoretical notions studied in mathematical chemistry,
and, especially, by the notions of distance-balanced graphs~\cite{ha,je} 
and the Mostar index of a graph~\cite{do},
Miklavi\v{c} and \v{S}parl~\cite{misp2} 
introduced the distance-unbalancedness of a graph.
Proving one of their conjectures, 
we showed~\cite{krra} that stars 
minimize the distance-unbalancedness among all trees of a fixed order.
For positive integers $n$ up to $15$,
Miklavi\v{c} and \v{S}parl also determined the trees that
maximize the distance-unbalancedness among all trees of a fixed order $n$.
Apart from the observation that all these trees were subdivided stars
their structure remained somewhat elusive.

In order to define distance-unbalancedness and explain our contribution,
we introduce some notation.
We consider only finite, simple, and undirected graphs.
For a graph $G$, and two vertices $u$ and $v$ of $G$, 
let ${\rm dist}_G(u,v)$ denote the {\it distance in $G$ between $u$ and $v$}, and 
let $n_G(u,v)$ be the number of vertices $w$ of $G$ 
that are closer to $u$ than to $v$,
that is, that satisfy ${\rm dist}_G(u,w)<{\rm dist}_G(v,w)$.
The {\it Mostar index}~\cite{do} of $G$ is
$${\rm Mo}(G)=\sum\limits_{uv\in E(G)}|n_G(u,v)-n_G(v,u)|.$$
A graph $G$ is {\it distance-balanced}~\cite{ha,je} 
if $n_G(u,v)=n_G(v,u)$ for every edge $uv$ of $G$;
or, equivalently, if ${\rm Mo}(G)=0$.
The {\it distance-unbalancedness}~\cite{misp2} of $G$ is
\begin{eqnarray}\label{e1}
{\rm uB}(G)&=&\sum\limits_{\{ u,v\}\in {V(G)\choose 2}}|n_G(u,v)-n_G(v,u)|,
\end{eqnarray}
where ${V(G)\choose 2}$ denotes the set of 
all $2$-element subsets of the vertex set $V(G)$ of $G$,
that is, the edge set of the complete graph with vertex set $V(G)$.
A graph $G$ is {\it highly distance-balanced}~\cite{misp1} 
if $n_G(u,v)=n_G(v,u)$ for every two distinct vertices $u$ and $v$ of $G$;
or, equivalently, if ${\rm uB}(G)=0$.
For a detailed discussion about the role of the above notions in mathematical chemistry,
we refer to the cited references.

In~\cite{misp2} Miklavi\v{c} and \v{S}parl reported the computational result that
for each integer $n$ up to $15$, 
the trees 
maximizing ${\rm uB}(T)$ 
among all trees $T$ of fixed order $n$ 
are all subdivided stars.
For general $n$,
they posed the problem to determine the subdivided stars
that have the largest distance-unbalancedness 
among all subdivided stars of order $n$, cf.~Problem 4.7 in~\cite{misp2}.
For positive integers $n_1,\ldots,n_k$,
let the {\it subdivided star $S(n_1,\ldots,n_k)$}
arise from the star $K_{1,k}$ with the $k$ edges $e_1,\ldots,e_k$
by subdividing the edge $e_i$ exactly $n_i-1$ times for every $i\in \{ 1,\ldots,k\}$,
that is, $S(n_1,\ldots,n_k)$ has order $1+n_1+\cdots+n_k$,
and removing its center vertex 
yields a forest whose components are $k$ paths of orders $n_1,\ldots,n_k$.
Representing $S(n_1,\ldots,n_k)$ by the tuple $(n_1,\ldots,n_k)$,
the following tuples encode the subdivided stars 
maximizing the distance-unbalancedness
among the trees of fixed order $n$ for $5\leq n\leq 15$~\cite{misp2}:

\medskip

\begin{center}
\begin{tabular}{cccccccc}
$(2,1,1)$ & 
$(2,2,1)$ &
$(2,2,1,1)$ & 
$(2,2,2,1)$ &  
$(3,2,2,1)$ & 
$(3,2,2,2)$ & 
$(3,3,2,1)$ & 
$(3,2,2,1,1)$ \\ 
$(3,3,2,2)$ & 
$(3,3,2,2,1)$ & 
$(3,3,3,2,1)$ & 
$(3,3,3,2,2)$ &
$(4,3,3,2,2)$.
\end{tabular}
\end{center}

\medskip

For $16\leq n\leq 59$, we determined the following tuples $(n_1,\ldots,n_k)$
encoding the subdivided stars 
maximizing the distance-unbalancedness
among the subdivided stars of order $n$:

\medskip

\begin{center}
\begin{tabular}{ccccc}
(4,3,3,2,2,1) &
(4,3,3,3,2,1) &
(4,4,3,3,2,1) &
(4,4,3,3,2,2) &
(4,4,3,3,3,2) \\
(4,4,4,3,3,2) &
(4,4,4,3,3,2,1) &
(4,4,4,3,3,2,2) &
(5,4,4,3,3,2,2) &
(5,4,4,3,3,3,2) \\
(5,4,4,4,3,3,2) &
(5,5,4,4,3,3,2) &
(5,5,4,4,3,3,2,1) &
(5,5,4,4,3,3,2,2) &
(5,5,4,4,4,3,2,2) \\ 
(5,5,4,4,4,3,3,2) &
(5,5,5,4,4,3,3,2) &
(6,5,5,4,4,3,3,2) &
(6,5,5,4,4,4,3,2) &
(6,5,5,5,4,4,3,2) \\
(6,5,5,5,4,4,3,3) &
(6,5,5,5,4,4,3,3,1) &
(6,5,5,5,4,4,3,3,2) &
(6,6,5,5,4,4,3,3,2) &
(6,6,5,5,4,4,4,3,2) \\
(6,6,5,5,5,4,4,3,2) &
(6,6,5,5,5,4,4,3,3) &
(6,6,6,5,5,4,4,3,3) &
(6,6,6,5,5,4,4,3,3,1) &
(6,6,5,5,5,4,4,3,3,2) \\ 
(6,6,6,5,5,4,4,3,3,2) &
(6,6,6,5,5,4,4,4,3,2) &
(6,6,6,5,5,5,4,4,3,2) &
(7,6,6,5,5,5,4,4,3,2) &
(7,6,6,6,5,5,4,4,3,2) \\
(7,6,6,6,5,5,4,4,3,3) &
(7,6,6,6,5,5,5,4,3,3) &
(7,7,6,6,5,5,5,4,3,3) &
(7,7,6,6,5,5,5,4,4,3) &
(7,7,6,6,6,5,5,4,4,3) \\
(7,7,6,6,5,5,5,4,4,3,2) &
(7,7,6,6,6,5,5,4,4,3,2) &
(7,7,7,6,6,5,5,4,4,3,2) & 
(7,7,6,6,6,5,5,4,4,3,3) & 
(7,7,7,6,6,5,5,5,4,3,2) \\ 
(7,7,7,6,6,5,5,4,4,3,3) & 
(7,7,7,6,6,5,5,5,4,3,3)
\end{tabular}
\end{center}

\medskip

There are some obvious qualitative features:
\begin{itemize}
\item The number $k$ of branches grows (slowly) with $n$.
\item All $n_i$ are less $n/2$.
\item The values of the $n_i$ range somewhat linearly from small to medium values.
\item With few exceptions for $n\in \{ 10,44,57,58\}$, 
the maximizing subdivided star is unique for a given order $n$.
\end{itemize}
We were unable though to determine some quantifiable regularity.

For a tree $T$ of order $n$, 
equation (\ref{e1}) implies that ${\rm uB}(T)$
is the sum of ${n\choose 2}$ terms that are all at most $n-2$,
in particular, 
\begin{eqnarray}\label{e2}
{\rm uB}(T)\leq \frac{n(n-1)(n-2)}{2}.
\end{eqnarray}
Our main result in the present paper implies the following.
\begin{corollary}\label{corollary1}
$$\lim\limits_{n\to \infty}\frac{\max\Big\{{\rm uB}(T):T\mbox{ is a tree of order }n\Big\}}{n^3}
=\frac{1}{2}.$$
\end{corollary}
In other words, while we do not determine the maximum considered in Corollary \ref{corollary1} exactly for every $n$, 
we at least show that it equals $\frac{n^3}{2}+o(n^3)$
by determining its leading term. 

Our main result is the following contribution to Problem 4.7 in~\cite{misp2}.
\begin{theorem}\label{theorem1}
$$\max\Big\{{\rm uB}(S(n_1,\ldots,n_k)):1+n_1+\cdots+n_k=n\Big\}
=\left(\frac{1}{2}-\frac{5}{6k}+\frac{1}{3k^2}\right)n^3+O(kn^2).$$
\end{theorem}
As discussed in \cite{krra},
distance-unbalancedness is a parameter that is hard to work with.
Many of the usual local arguments do not work well,
because its definition involves pairs of vertices at all possible distances.
In \cite{krra}, 
we overcame this difficulty by `localizing' our approach;
more precisely, 
we considered some carefully chosen alternative parameters 
that were more amenable to local arguments.
In the present paper, our approach is different;
we approximate (\ref{e1}) by smooth functions defined by suitable integrals, 
and obtain our result by solving a smooth optimization problem.
Intuitively, this approach corresponds to fixing the number $k$ of branches
and letting $n$ go to infinity.
As it turns out, cf.~Lemma \ref{lemma2} below, 
in order to maximize distance-unbalancedness in this setting, 
one should distribute the vertices as evenly to the different branches as possible,
that is, we are not able to reproduce the somewhat linear increase of the $n_i$
from small to medium values observed in the computational data,
where $k$ is allowed to grow with $n$.

All proofs are given in the next section.

\section{Proofs}

Let $n_1,\ldots,n_k$ be positive integers such that $n_1\geq n_2\geq \ldots \geq n_k$.
For $k=2$, the subdivided spider $S(n_1,n_2)$ is just a path $P_{1+n_1+n_2}$, 
whose distance-unbalancedness satisfies Theorem \ref{theorem1},
cf.~Proposition 3.3 in \cite{misp2}.
Hence, let $k\geq 3$.
Let $n=1+n_1+\cdots +n_k$, and let $S=S(n_1,\ldots,n_k)$.

The distance-unbalancedness ${\rm uB}(S)$ of $S$ satisfies
\begin{eqnarray*}
{\rm uB}(n_1,\ldots,n_k) & = & 
{\rm uB}_1(n_1,\ldots,n_k)
+{\rm uB}_2(n_1,\ldots,n_k)
+{\rm uB}_3(n_1,\ldots,n_k)
+{\rm uB}_4(n_1,\ldots,n_k),\mbox{ where}\\
{\rm uB}_1(n_1,\ldots,n_k) & = & 
\sum_{i=1}^k\sum_{d=1}^{n_i}|n-2n_i-1+d|,\\
{\rm uB}_2(n_1,\ldots,n_k) & = & 
\sum_{i=1}^{k-1}\sum_{j=i+1}^k(n_i-n_j)n_j,\\
{\rm uB}_3(n_1,\ldots,n_k) & = & 
\sum_{i=1}^k\sum_{d=1}^{n_i-1}\sum_{d'=1}^{n_i-d}|n-2n_i-1+2d+d'|,\mbox{ and}\\
{\rm uB}_4(n_1,\ldots,n_k) & = &
\sum_{i=1}^{k-1}\sum_{j=i+1}^k
\left(
\sum_{d_j=1}^{\min\{ n_j,n_i-1\}}\sum_{d_i=d_j+1}^{n_i}|n-2n_i-1+d_i-d_j|
+
\sum_{d_i=1}^{n_j-1}\sum_{d_j=d_i+1}^{n_j}|n-2n_j-1+d_j-d_i|
\right). 
\end{eqnarray*}
The four terms above collect the contribution to ${\rm uB}(S)$ 
generated by different types of pairs $\{ u,v\}$ of distinct vertices:
\begin{itemize}
\item The term $|n-2n_i-1+d|$ in ${\rm uB}_1(n_1,\ldots,n_k)$
is the contribution generated by a pair containing
the center vertex $c$ of $S$ and a vertex at distance $d$ from $c$ in the $i$th branch.
\item The term $(n_i-n_j)n_j$ in ${\rm uB}_2(n_1,\ldots,n_k)$
is the contribution generated by all pairs containing
one vertex in the $i$th branch 
and one vertex in the $j$th branch 
both at the same distance from $c$.
Note that there are $n_j$ choices for this distance.
\item The term $|n-2n_i-1+2d+d'|$ in ${\rm uB}_3(n_1,\ldots,n_k)$
is the contribution generated by a pair containing
two vertices in the $i$th branch,
one at distance $d$ and one at distance $d+d'$ from $c$.
\item The two terms $|n-2n_i-1+d_i-d_j|$ and $|n-2n_j-1+d_j-d_i|$
in ${\rm uB}_4(n_1,\ldots,n_k)$
are the contributions generated by pairs containing
one vertex in the $i$th branch at distance $d_i$ from $c$
and one vertex in the $j$th branch at distance $d_j$ from $c$
with $d_i>d_j$ for the first term and $d_i<d_j$ for the second term, respectively.
\end{itemize}
For $x_1,\ldots,x_k\in [0,1]$ 
with $x_1\geq x_2\geq \ldots \geq x_k$,
we consider the following two integrals as continuous variants of 
${\rm uB}_3$ and ${\rm uB}_4$,
where $x_i$ corresponds to $\frac{n_i}{n}$:
\begin{eqnarray*}
\widetilde{\rm uB}_3(x_1,\ldots,x_k) & = & 
\sum_{i=1}^k
\int_0^{x_i}\int_0^{x_i-y}|1-2x_i+2y+y'| dy' dy\mbox{ and}\\
\widetilde{\rm uB}_4(x_1,\ldots,x_k) & = & 
\sum_{i=1}^{k-1}\sum_{j=i+1}^k
\left(
\int_0^{x_j}\int_{y_j}^{x_i}|1-2x_i+y_i-y_j|dy_i dy_j
+
\int_0^{x_j}\int_{y_i}^{x_j}|1-2x_j+y_j-y_i|dy_j dy_i
\right). 
\end{eqnarray*}
Our first lemma quantifies in which sense these expressions approximate ${\rm uB}(S)$.

\begin{lemma}\label{lemma1}
If $x_i=\frac{n_i}{n}$ for every $i\in \{ 1,\ldots,k\}$, then
$$
\left|{\rm uB}(S)-n^3\left(\widetilde{\rm uB}_3(x_1,\ldots,x_k)+\widetilde{\rm uB}_4(x_1,\ldots,x_k)\right)\right|\leq O(kn^2).$$
\end{lemma}
\begin{proof}
Note that
\begin{eqnarray}
{\rm uB}_1(n_1,\ldots,n_k) 
& = & \sum_{i=1}^k\sum_{d=1}^{n_i}|n-2n_i-1+d|
\leq \sum_{i=1}^knn_i
\leq n^2\mbox{ and}\label{esmall1}\\
{\rm uB}_2(n_1,\ldots,n_k) 
& = & \sum_{i=1}^{k-1}\sum_{j=i+1}^k(n_i-n_j)n_j
\leq \sum_{i=1}^{k-1}n_i\sum_{j=i+1}^kn_j
\leq \sum_{i=1}^{k-1}n_in
\leq n^2,\label{esmall2}
\end{eqnarray}
Note that $x_1+\cdots+x_k=\frac{n-1}{n}<1$.

First, we consider the case $2n_1\leq n$, which implies $x_1\leq \frac{1}{2}$.
Note that, in this case, 
all absolute values ``$|\cdot|$'' in the above expressions are redundant.

Since
\begin{eqnarray*}
&& \left|\sum_{d=1}^{n_i-1}\sum_{d'=1}^{n_i-d}(n-2n_i-1+2d+d')
-n^3\int_0^{x_i}\int_0^{x_i-y}(1-2x_i+2y+y' ) dy' dy\right|\\
&=& \left|\sum_{d=1}^{x_in-1}\sum_{d'=1}^{x_in-d}(n-2x_in-1+2d+d')
-n^3\int_0^{x_i}\int_0^{x_i-y}(1-2x_i+2y+y' ) dy' dy\right|\\
&=& \left|\frac{n^2x_i(1-x_i)(x_in-1)}{2}-\frac{n^3x_i^2(1-x_i)}{2}\right|\\
&=& \frac{n^2x_i(1-x_i)}{2},
\end{eqnarray*}
we obtain
\begin{eqnarray*}
\left|{\rm uB}_3(n_1,\ldots,n_k)-n^3\widetilde{\rm uB}_3(x_1,\ldots,x_k)\right|
&\leq & \sum_{i=1}^k\frac{n^2x_i(1-x_i)}{2}
\leq \frac{n^2}{2}\sum_{i=1}^kx_i
\leq \frac{n^2}{2}.
\end{eqnarray*}
Similarly, since
\begin{eqnarray*}
&& \Bigg|\left(\sum_{d_j=1}^{\min\{ n_j,n_i-1\}}\sum_{d_i=d_j+1}^{n_i}(n-2n_i-1+d_i-d_j)
+
\sum_{d_i=1}^{n_j-1}\sum_{d_j=d_i+1}^{n_j}(n-2n_j-1+d_j-d_i)\right)\\
&& \,\,\,\,\,\,\,\,\,\,\,\,\,\,\,\,\,\,- n^3\left(
\int_0^{x_j}\int_{y_j}^{x_i}(1-2x_i+y_i-y_j)dy_i dy_j
+
\int_0^{x_j}\int_{y_i}^{x_j}(1-2x_j+y_j-y_i)dy_j dy_i
\right)\Bigg|\\
&=& 
\left|
\frac{1}{6}x_jn\left(
3x_ix_jn^2
+6x_in^2
+6x_jn
+4
-9x_i^2n^2
-4x_j^2n^2
-6n
\right)
-n^3\left(\frac{x_ix_j(x_j+2)}{2}-\frac{3x_i^2x_j}{2}-\frac{2x_j^3}{3}\right)
\right|\\
& = & x_jn\left((1-x_j)n-\frac{2}{3}\right),
\end{eqnarray*}
we obtain
\begin{eqnarray*}
\left|{\rm uB}_4(n_1,\ldots,n_k)-n^3\widetilde{\rm uB}_4(x_1,\ldots,x_k)\right|
&\leq & \sum_{i=1}^{k-1}\sum_{j=i+1}^kx_jn\left((1-x_j)n-\frac{2}{3}\right)
\leq n^2\sum_{i=1}^{k-1}\sum_{j=i+1}^kx_j
=n^2\sum_{i=1}^k(i-1)x_i
\leq kn^2.
\end{eqnarray*}
Now, using the above estimates, (\ref{esmall1}), and (\ref{esmall2}),
\begin{eqnarray*}
&&\left|{\rm uB}(S)-n^3\left(\widetilde{\rm uB}_3(x_1,\ldots,x_k)+\widetilde{\rm uB}_4(x_1,\ldots,x_k)\right)\right|\\
&\leq &
{\rm uB}_1(n_1,\ldots,n_k)+{\rm uB}_2(n_1,\ldots,n_k)\\
&&+\left|{\rm uB}_3(n_1,\ldots,n_k)-n^3\widetilde{\rm uB}_3(x_1,\ldots,x_k)\right|
+\left|{\rm uB}_4(n_1,\ldots,n_k)-n^3\widetilde{\rm uB}_4(x_1,\ldots,x_k)\right|\\
&\leq & n^2+n^2+\frac{n^2}{2}+kn^2.
\end{eqnarray*}
Next, we consider the case $2n_1>n$, which implies $x_1>\frac{1}{2}$.
Since $n_1+\cdots+n_k=n-1$, we still have $2n_i<n$ and $x_i<\frac{1}{2}$ 
for $i\geq 2$.
Correctly resolving 
the absolute values ``$|\cdot|$'' in the expressions involving $n_1$ and $x_1$,
and adapting the above calculations in a straightforward way,
allows to complete the proof.
For details, we refer to the appendix.
\end{proof}
Our next lemma concerns the smooth optimization problem
mentioned in the introduction.

Let 
$$f(x_1,\ldots,x_k)=\widetilde{\rm uB}_3(x_1,\ldots,x_k)+\widetilde{\rm uB}_4(x_1,\ldots,x_k).$$
\begin{lemma}\label{lemma2}
For $x_1,\ldots,x_k\in [0,1]$
with $x_1\geq x_2\geq \ldots \geq x_k$ and $x_1+\cdots+x_k\leq 1$, we have
$$f(x_1,\ldots,x_k)\leq f\left(\frac{1}{k},\ldots,\frac{1}{k}\right)=\frac{1}{2}-\frac{5}{6k}+\frac{1}{3k^2}.$$
\end{lemma}
\begin{proof}
Let $(x_1,\ldots,x_k)$ maximize the continuous function $f(x_1,\ldots,x_k)$
subject to the constraints
$x_1,\ldots,x_k\in [0,1]$
with $x_1\geq x_2\geq \ldots \geq x_k$ and $x_1+\cdots+x_k\leq 1$,
which define a compact subset of $\mathbb{R}^k$.
Our goal is to show that each $x_i$ equals $\frac{1}{k}$.
Similarly as in the proof of Lemma \ref{lemma1},
we first consider the case that $x_1\leq \frac{1}{2}$;
in view of the desired result, this is actually the more relevant case
leading to the maximum value of $f$.
In this case, we have
\begin{eqnarray*}
f(x_1,\ldots,x_k)&=&\widetilde{\rm uB}_3(x_1,\ldots,x_k)+\widetilde{\rm uB}_4(x_1,\ldots,x_k)\\
&=&\sum_{i=1}^k
\int_0^{x_i}\int_0^{x_i-y}(1-2x_i+2y+y' ) dy' dy\\
&&+\sum_{i=1}^{k-1}\sum_{j=i+1}^k
\left(
\int_0^{x_j}\int_{y_j}^{x_i}(1-2x_i+y_i-y_j)dy_i dy_j
+
\int_0^{x_j}\int_{y_i}^{x_j}(1-2x_j+y_j-y_i)dy_j dy_i
\right)\\
&=&\sum_{i=1}^k\left(\frac{x_i^2}{2}-\frac{x_i^3}{2}\right)
+\sum_{i=1}^{k-1}\sum_{j=i+1}^k
\left(x_ix_j+\frac{x_ix_j^2}{2}-\frac{3x_i^2x_j}{2}-\frac{2x_j^3}{3}
\right).
\end{eqnarray*}
First, suppose, for a contradiction, that the $x_i$ are not all equal,
and that $i\in \{ 1,\ldots,k-1\}$ is the smallest index with $x_i>x_{i+1}$.
For a suitable function $g$, we have
\begin{eqnarray*}
f(x_1,\ldots,x_k) & = & g(x_1,\ldots,x_{i-1},x_{i+2},\ldots,x_k)\\
&&+\frac{1}{2}x_i^2
-\frac{1}{2}x_i^3\\
&&
+\frac{1}{2}x_{i+1}^2
-\frac{1}{2}x_{i+1}^3\\
&&+a_1x_i+\frac{1}{2}a_1x_i^2-\frac{3}{2}a_2x_i-\frac{2}{3}(i-1)x_i^3
+b_1x_i+\frac{1}{2}b_2x_i-\frac{3}{2}b_1x_i^2-\frac{2}{3}b_3\\
&&+a_1x_{i+1}+\frac{1}{2}a_1x_{i+1}^2-\frac{3}{2}a_2x_{i+1}-\frac{2}{3}(i-1)x_{i+1}^3
+b_1x_{i+1}+\frac{1}{2}b_2x_{i+1}-\frac{3}{2}b_1x_{i+1}^2-\frac{2}{3}b_3\\
&&+x_ix_{i+1}+\frac{x_ix_{i+1}^2}{2}-\frac{3x_i^2x_{i+1}}{2}-\frac{2x_{i+1}^3}{3},
\end{eqnarray*}
where
$$a_\ell = \sum_{j=1}^{i-1}x_j^\ell\,\,\,\,\,\,\mbox{ and }\,\,\,\,\,\,
b_\ell = \sum_{j=i+2}^kx_j^\ell.$$
By the choice of $i$, we have $a_\ell=(i-1)x_i^\ell$,
and it follows that
\begin{eqnarray*}
f(x_1,\ldots,x_i-\epsilon,x_{i+1}+\epsilon,\ldots,x_k)-f(x_1,\ldots,x_k)
&=& (x_i-x_{i+1})\Big((i-1)x_i+2(i+1)x_{i+1}+3b_1\Big)\epsilon+O(\epsilon^2),
\end{eqnarray*}
which implies the contradiction
$$f(x_1,\ldots,x_i-\epsilon,x_{i+1}+\epsilon,\ldots,x_k)>f(x_1,\ldots,x_k)$$
for sufficiently small $\epsilon>0$.
Hence, all $x_i$ are equal to some $x$ with $0\leq x\leq \frac{1}{k}\leq \frac{1}{3}$,
which implies 
$$
f(x_1,\ldots,x_k)
=k\left(\frac{x^2}{2}-\frac{x^3}{2}\right)
+{k\choose 2}\left(x^2+\frac{x^3}{2}-\frac{3x^3}{2}-\frac{2x^3}{3}\right)
=\frac{k^2(3x^2-5x^3)}{6}+\frac{kx^3}{3}.$$
Since the last expression is increasing in $x$ for $x<\frac{2}{5}$, 
it follows that $x=\frac{1}{k}$, 
which completes the proof in this case.

Now, we consider the case $x_1\geq\frac{1}{2}$.
Correctly resolving 
the absolute values ``$|\cdot|$'' in the expressions involving $x_1$,
and adapting the above calculations in a straightforward way,
we obtain that $x_1=\frac{1}{2}$
and $x_2=\ldots=x_k=\frac{1}{2(k-1)}$,
that is, we return to the above case $x_1\leq\frac{1}{2}$,
which completes the proof.
For details, we refer to the appendix.
\end{proof}
Lemma \ref{lemma1} and Lemma \ref{lemma2} together 
immediately imply Theorem \ref{theorem1}, 
and choosing $k=\lfloor\sqrt{n}\rfloor$ in Theorem \ref{theorem1}
immediately yields Corollary \ref{corollary1}.

\pagebreak

\begin{center}
{\bf Appendix}
\end{center}
{\it Proof of Lemma \ref{lemma1}.}
We give details for the case $2n_1>n$.

Note that 
$x_1>\frac{1}{2}$,
$n_2+\cdots +n_k<\frac{n}{2}$ 
and 
$x_2+\cdots+x_k<\frac{1}{2}$.

We have
\begin{eqnarray*}
&&\sum_{d=1}^{n_1-1}\sum_{d'=1}^{n_1-d}|n-2n_1-1+2d+d'|\\
& = & 
\sum_{d=1}^{\left\lfloor n_1-\frac{n}{2}\right\rfloor}\left(\sum_{d'=1}^{2n_1+1-n-2d}\Big((2n_1+1-n)-(2d+d')\Big)
+
\sum_{d'=2n_1+1-n-2d+1}^{n_1-d}\Big(-(2n_1+1-n)+(2d+d')\Big)\right)\\
&&+\sum_{d=\left\lfloor n_1-\frac{n}{2}\right\rfloor+1}^{n_1-1}
\sum_{d'=1}^{n_1-d}\Big(-(2n_1+1-n)+(2d+d')\Big)\\
& = & 
\left(\frac{5x_1^3}{6}-\frac{3x_1^2}{2}+x_1-\frac{1}{6}\right)n^3
+\left(-\frac{x_1^2}{2}+\frac{x_1}{2}-\frac{1}{4}\right)n^2
+\left(-\frac{x_1}{3}+\frac{1}{6}\right)n
+\left\{
\begin{array}{ll}
0, & \mbox{ $n$ even, and}\\
\frac{1}{4}, & \mbox{ $n$ odd.}
\end{array}\right.
\end{eqnarray*}
and 
\begin{eqnarray*}
&& \int_0^{x_1}\int_0^{x_1-y}|1-2x_1+2y+y'| dy' dy\\
&=& 
\int_0^{x_1-\frac{1}{2}}\left(
\int_0^{2x_1-1-2y}\Big((2x_1-1)-(2y+y')|\Big)dy' dy
+
\int_{2x_1-1-2y}^{x_1-y}\Big(-(2x_1-1)+(2y+y')|\Big)dy' dy\right)\\
&&+
\int_{x_1-\frac{1}{2}}^{x_1}
\int_{0}^{x_1-y}\Big(-(2x_1-1)+(2y+y')|\Big)dy' dy\\
&=& \frac{5x_1^3}{6}-\frac{3x_1^2}{2}+x_1-\frac{1}{6}.
\end{eqnarray*}
This allows to deduce
\begin{eqnarray*}
&& \left|\sum_{d=1}^{n_1-1}\sum_{d'=1}^{n_1-d}|n-2n_1-1+2d+d'|
-n^3\int_0^{x_1}\int_0^{x_1-y}|1-2x_1+2y+y'| dy' dy\right|
\leq \frac{n^2+1}{4},
\end{eqnarray*}
which implies
\begin{eqnarray*}
\left|{\rm uB}_3(n_1,\ldots,n_k)-n^3\widetilde{\rm uB}_3(x_1,\ldots,x_k)\right|
&\leq & \frac{n^2+1}{4}+\sum_{i=2}^k\frac{n^2x_i(1-x_i)}{2}\\
&\leq &\frac{n^2+1}{4}+\frac{n^2}{2}\sum_{i=2}^kx_i\\
&\leq &\frac{2n^2+1}{4}.
\end{eqnarray*}
Similarly, we have
\begin{eqnarray*}
&&\sum_{d_j=1}^{n_j}\sum_{d_1=d_j+1}^{n_1}|n-2n_1-1+d_1-d_j|
+
\sum_{d_1=1}^{n_j-1}\sum_{d_j=d_1+1}^{n_j}|n-2n_j-1+d_j-d_1|\\
&=&
\sum_{d_j=1}^{n_j}
\left(\sum_{d_1=d_j+1}^{2n_1+1-n+d_j}\Big((2n_1+1-n)-(d_1-d_j)\Big)
+
\sum_{d_1=2n_1+1-n+d_j+1}^{n_1}\Big(-(2n_1+1-n)+(d_1-d_j)\Big)\right)\\
&&+
\sum_{d_1=1}^{n_j-1}\sum_{d_j=d_1+1}^{n_j}(n-2n_j-1+d_j-d_1)\\
&=&
\left(\frac{5}{2}x_1^2x_j
+\frac{1}{2}x_1x_j^2
-3x_1x_j
+x_j
-\frac{2}{3}x_j^3\right)n^3
+
\left(
2x_1n
+x_jn
-2n
+\frac{2}{3}\right)x_jn
\end{eqnarray*}
and
\begin{eqnarray*}
&&\int_0^{x_j}\int_{y_j}^{x_1}|1-2x_1+y_1-y_j|dy_1 dy_j
+
\int_0^{x_j}\int_{y_1}^{x_j}|1-2x_j+y_j-y_1|dy_j dy_1\\
&=& 
\int_0^{x_j}
\left(
\int_{y_j}^{2x_1-1+y_j}\Big((2x_1-1)-(y_1-y_j)\Big)dy_i dy_j
+
\int_{2x_1-1+y_j}^{x_1}\Big(-(2x_1-1)+(y_1-y_j)\Big)dy_i dy_j
\right)\\
&&+\int_0^{x_j}\int_{y_1}^{x_j}(1-2x_j+y_j-y_1)dy_j dy_1\\
&=& 
\frac{5}{2}x_1^2x_j
+\frac{1}{2}x_1x_j^2
-3x_1x_j
+x_j
-\frac{2}{3}x_j^3.
\end{eqnarray*}
This implies
\begin{eqnarray*}
&& \sum_{j=2}^k
\Bigg|
\left(\sum_{d_j=1}^{n_j}\sum_{d_1=d_j+1}^{n_1}|n-2n_1-1+d_1-d_j|
+
\sum_{d_1=1}^{n_j-1}\sum_{d_j=d_1+1}^{n_j}|n-2n_j-1+d_j-d_1|\right)\\
&&\,\,\,\,\,\,\,\,\,\,
-\left(
\int_0^{x_j}\int_{y_j}^{x_1}|1-2x_1+y_1-y_j|dy_1 dy_j
+
\int_0^{x_j}\int_{y_1}^{x_j}|1-2x_j+y_j-y_1|dy_j dy_1
\right)
\Bigg|\\
&=& \sum_{j=2}^k 2nx_j\left(\left(1-x_1-\frac{x_j}{2}\right)n+\frac{1}{3}\right)\\
&\leq & 2n\left(n+\frac{1}{3}\right)\sum_{j=2}^kx_j\\
&\leq & n\left(n+\frac{1}{3}\right),
\end{eqnarray*}
and, hence,
\begin{eqnarray*}
\left|{\rm uB}_4(n_1,\ldots,n_k)-n^3\widetilde{\rm uB}_4(x_1,\ldots,x_k)\right|
&\leq & n\left(n+\frac{1}{3}\right)+\sum_{i=2}^{k-1}\sum_{j=i+1}^kx_jn\left((1-x_j)n-\frac{2}{3}\right)\\
&\leq & n\left(n+\frac{1}{3}\right)+n^2\sum_{i=2}^{k-1}\sum_{j=i+1}^kx_j\\
&=& n\left(n+\frac{1}{3}\right)+n^2\sum_{i=2}^k(i-2)x_i\\
& \leq & n\left(n+\frac{1}{3}\right)+\frac{k}{2}n^2.
\end{eqnarray*}
Now, using the above estimates together with (\ref{esmall1}) and (\ref{esmall2}),
we obtain
\begin{eqnarray*}
&&\left|{\rm uB}(S)-n^3\left(\widetilde{\rm uB}_3(x_1,\ldots,x_k)+\widetilde{\rm uB}_4(x_1,\ldots,x_k)\right)\right|\\
&\leq &
{\rm uB}_1(n_1,\ldots,n_k)+{\rm uB}_2(n_1,\ldots,n_k)\\
&&+\left|{\rm uB}_3(n_1,\ldots,n_k)-n^3\widetilde{\rm uB}_3(x_1,\ldots,x_k)\right|
+\left|{\rm uB}_4(n_1,\ldots,n_k)-n^3\widetilde{\rm uB}_4(x_1,\ldots,x_k)\right|\\
&\leq & n^2+n^2+\frac{2n^2+1}{4}+n\left(n+\frac{1}{3}\right)+\frac{k}{2}n^2\\
& \leq & O(kn^2),
\end{eqnarray*}
which completes the proof.
$\Box$

\pagebreak

\noindent {\it Proof of Lemma \ref{lemma2}.}
We give details for the case $x_1>\frac{1}{2}$.
Note that $x_1>x_2$.

Exploiting previous calculations, we obtain
\begin{eqnarray*}
f(x_1,\ldots,x_k) & = & 
\left(x_1-\frac{3}{2}x_1^2+\frac{5}{6}x_1^3-\frac{1}{6}\right)
+\sum_{i=2}^k\left(\frac{1}{2}x_i^2-\frac{1}{2}x_i^3\right)\\
&&+\sum_{j=2}^k\left(\frac{5}{2}x_1^2x_j+\frac{1}{2}x_1x_j^2-3x_1x_j+x_j-\frac{2}{3}x_j^3\right)\\
&&+\sum_{i=2}^{k-1}\sum_{j=i+1}^k\left(x_ix_j+\frac{1}{2}x_ix_j^2-\frac{3}{2}x_i^2x_j-\frac{2}{3}x_j^3\right).
\end{eqnarray*}
Our first goal is to show that $x_2=\ldots=x_k$.
Therefore, similarly as above,
suppose, for a contradiction, that these $x_i$ are not all equal,
and that $i\in \{ 2,\ldots,k-1\}$ is the smallest index with $x_i>x_{i+1}$.
Again, for a suitable function $\hat{g}$, we have
\begin{eqnarray*}
f(x_1,\ldots,x_k) & = & \hat{g}(x_1,\ldots,x_{i-1},x_{i+2},\ldots,x_k)\\
&&+\frac{1}{2}x_i^2
-\frac{1}{2}x_i^3\\
&&
+\frac{1}{2}x_{i+1}^2
-\frac{1}{2}x_{i+1}^3\\
&&+\Big(a_1-x_i-3x_1+1\Big)x_i
+\frac{1}{2}\Big(a_1-x_i+x_1\Big)x_i^2
-\frac{3}{2}\Big(a_2-x_i^2-\frac{5}{3}x_1^2\Big)x_i
-\frac{2}{3}(i-1)x_i^3\\
&&+b_1x_i+\frac{1}{2}b_2x_i-\frac{3}{2}b_1x_i^2-\frac{2}{3}b_3\\
&&+\Big(a_1-x_{i}-3x_1+1\Big)x_{i+1}
+\frac{1}{2}\Big(a_1-x_{i}+x_1\Big)x_{i+1}^2
-\frac{3}{2}\Big(a_2-x_{i}^2-\frac{5}{3}x_1^2\Big)x_{i+1}
-\frac{2}{3}(i-1)x_{i+1}^3\\
&&+b_1x_{i+1}+\frac{1}{2}b_2x_{i+1}-\frac{3}{2}b_1x_{i+1}^2-\frac{2}{3}b_3\\
&&+x_ix_{i+1}+\frac{x_ix_{i+1}^2}{2}-\frac{3x_i^2x_{i+1}}{2}-\frac{2x_{i+1}^3}{3},
\end{eqnarray*}
where $a_\ell=(i-1)x_i^\ell$
and $b_\ell=\sum\limits_{j=i+2}^kx_j^\ell$.

We obtain
\begin{eqnarray*}
&& f(x_1,\ldots,x_i-\epsilon,x_{i+1}+\epsilon,\ldots,x_k)-f(x_1,\ldots,x_k)\\
&=& 
\underbrace{\left(
\left(
1-x_1+\left(i+\frac{1}{2}\right)x_i+\left(2i+\frac{5}{2}\right)x_{i+1}+3b_1
\right)(x_i-x_{i+1})
+x_{i+1}^2+2x_{i+1}
\right)}_{>0}\epsilon+O(\epsilon^2),
\end{eqnarray*}
which implies the contradiction
$$f(x_1,\ldots,x_i-\epsilon,x_{i+1}+\epsilon,\ldots,x_k)>f(x_1,\ldots,x_k)$$
for sufficiently small $\epsilon>0$.
Hence, we have $x_2=\ldots=x_k=y$
for some $y$ with $0\leq y\leq \frac{1-x_1}{k-1}<\frac{1}{4}$.

Let
\begin{eqnarray*}
f_2(x_1,y)&=&f(x_1,y\ldots,y)\\
&=& 
\left(x_1-\frac{3}{2}x_1^2+\frac{5}{6}x_1^3-\frac{1}{6}\right)
+(k-1)\left(\frac{1}{2}y^2-\frac{1}{2}y^3\right)\\
&&+(k-1)\left(\frac{5}{2}x_1^2y+\frac{1}{2}x_1y^2-3x_1y+y-\frac{2}{3}y^3\right)
+{k-1\choose 2}\left(y^2+\frac{1}{2}y^3-\frac{3}{2}y^3-\frac{2}{3}y^3\right)\\
&=& 
\left(-\frac{5}{6}k^2+\frac{4}{3}k-\frac{1}{2}\right)y^3
+
\left(\frac{(k-1)(k+x_1-1)}{2}\right)y^2\\
&&+
\left(\frac{(k-1)(5x_1^2-6x_1+2)}{2}\right)y
+
\left(x_1-\frac{3}{2}x_1^2+\frac{5}{6}x_1^3-\frac{1}{6}\right).
\end{eqnarray*}
For $y<\frac{2}{5}$, we have
\begin{eqnarray*}
\frac{\partial}{\partial y}f_2(x_1,y)
& = & \frac{k-1}{2}\left(\underbrace{(3-5k)}_{<0}y^2+(2k+2x_1-2)y+(5x_1^2-6x_1+2)\right)\\
& \stackrel{y<\frac{2}{5}}{\geq} & \frac{k-1}{2}\left((3-5k)\frac{2}{5}y+(2k+2x_1-2)y+(5x_1^2-6x_1+2)\right)\\
&=&  \frac{k-1}{2}\left(\underbrace{\left(2x_1-\frac{4}{5}\right)}_{>0\,\,for\,\,x_1\geq \frac{1}{2}}y+\underbrace{(5x_1^2-6x_1+2)}_{>0\,\,for\,\,x_1\geq \frac{1}{2}}\right)\\
&>& 0,  
\end{eqnarray*}
that is, $f_2(x_1,y)$ is increasing in $y\leq \frac{1-x_1}{k-1}<\frac{1}{4}$,
which implies $y=\frac{1-x_1}{k-1}$.

Let
\begin{eqnarray*}
f_3(x_1)&=&f_2\left(x_1,\frac{1-x_1}{k-1}\right)\\
& = & \frac{(-10x_1^3 + 27x_1^2 - 24x_1 + 8)k^2 + (28x_1^3 - 75x_1^2 + 66x_1 - 21)k - 16x_1^3 + 42x_1^2 - 36x_1 + 11}{6(k - 1)^2}.
\end{eqnarray*}
We have $\frac{d}{dx_1}f_3(x_1)=0$ 
only for $x_1\in \left\{ \frac{4k-3}{5k-4},1\right\}$,
where $\frac{1}{2}<\frac{4k-3}{5k-4}<1$.

Since 
$$\frac{d^2}{dx_1^2}f_3(x_1)\mid_{x_1=\frac{4k-3}{5k-4}}=6k^2-18k+12>0$$
and
$$f_3\left(\frac{1}{2}\right)-f_3\left(1\right)=\frac{2k^2-5k+2}{24(k-1)^2}>0,$$
it follows that 
$$f(x_1,\ldots,x_k)\leq \max\left\{ f_3(x_1):x_1\in \left[\frac{1}{2},1\right]\right\}=f_3\left(\frac{1}{2}\right)=\frac{(3k-2)(2k-3)}{24(k-1)^2}<\frac{1}{2}-\frac{5}{6k}+\frac{1}{3k^2},$$
which completes the proof.
$\Box$

\end{document}